\documentclass[a4paper,11pt]{amsart}
\usepackage[letterpaper, margin=1.2in]{geometry}
\usepackage{latexsym}
\usepackage{amssymb}
\usepackage{amsmath}
\usepackage{amsfonts}
\usepackage{color}
\newtheorem{thm}{THEOREM}[section]

\newtheorem{lem}[thm]{LEMMA}

\newtheorem{cor}[thm]{COROLLARY}

\theoremstyle{definition}

\newcommand{\N}{\mathbb N}
\newcommand{\Z}{\mathbb Z}
\newcommand{\Q}{\mathbb Q}

\newcommand{\m}{\mathfrak{m}}
\newcommand{\p}{\mathfrak{p}}

\def\al{\alpha}

\def\om{\omega}
\def\ol{\overline}

\newcounter{cs}
\stepcounter{cs}
\newcommand{\casos}{\begin{itemize}}
\newcommand{\fcasos}{\end{itemize}\setcounter{cs}{1}}

\newfont{\tit}{cmr12 scaled \magstep3}

\begin{document}

\title {On the Integral Closedness of $R[\al]$}

\author{Abdulaziz Deajim}
\address{{\it Department of Mathematics, King Khalid University, P.O. Box 9004, Abha, Saudi Arabia}} \email{deajim@kku.edu.sa, deajim@gmail.com}
\author{Lhoussain El Fadil}
\address{{\it Department of Mathematics, Faculty of Sciences Dhar-Mahraz,  B.P. 1796, Fes, Morocco}} \email{lhouelfadil2@gmail.com}

\date{\today}

\begin {abstract}
Let $R$ be a Dedekind ring, $K$ its quotient field, and $L=K(\al)$ a finite field extension of $K$ defined by a monic irreducible polynomial $f(x)\in R[x]$. We give an easy version of Dedekind's criterion which computationally improves those versions know in the literature. We further use this result to give a sufficient condition for the integral closedness of $R[\al]$ when $f(x)=x^n-a$. In case $R$ is a ring of integers of a number field, we give yet sufficient and necessary conditions for this to hold, generalizing and improving in both cases some known results in this direction. Some highlighting examples are also given.

\end {abstract}
\maketitle

\section{{\bf INTRODUCTION AND STATEMENTS OF MAIN RESULTS}}\label{intro}

For a complex number $\al$ integral over $\mathbb{Q}$, a criterion that tests the integral closedness of $\mathbb{Z}[\al]$ in the number field $\mathbb{Q}(\al)$ was given in the milestone paper \cite{Ded} of R. Dedekind (see also \cite{Cohen} or almost any book in algebraic number theory for a more modern treatment). As is well known, Dedekind's criterion utilizes the irreducible factorization of the reduction modulo prime integres of the minimal polynomial of $\al$. S. Khanduja and M. Kummar, in \cite{KK2}, gave a generalization of this criterion to extensions of Dedekind rings. Ershov, in \cite{Er}, gave yet a generalization of this criterion to extensions of rings of valuation. This criterion had, and still have, important applications in many relevant areas such as (but not limited to) the study of prime ideal factorizations in Dedekind rings, the computation of discriminants of number fields, and the existence of integral power bases in extensions of Dedekind rings (see for instance \cite{Cha-Dea}, \cite{KK1}, \cite{Sch}, and \cite{Uch}).

Let $(K,\nu)$ be a valued field with $\nu$ a rank-one discrete valuation, $R_\nu$ the ring of valuation of $\nu$, $\m_\nu$ the maximal ideal of $R_\nu$, $\pi$ a generator of $\m_\nu$, and $k_\nu=R_\nu/\m_\nu$ the residue field of $\nu$. We assume, by normalization if necessary, that $\nu(K^*)=\Z$ (so, in particular, $\nu(\pi)=1$). Denote also by $\nu$ the Gaussian extension of $\nu$ to the ring $R_\nu[x]$. Let $F(x) \in R_\nu[x]$ be a monic irreducible polynomial, $L=K(\al)$ the extension field of $K$ generated by a root $\al$ of $F$, and $S_\nu$ the integral closure of $R_\nu$ in $L$. Assume that $\overline{F}(x) \equiv \prod_{i=1}^r \overline{\phi}_i^{l_i}(x) \; (\mbox{mod}\; \m_\nu)$ is the monic irreducible factorization of $\overline{F}$ in $k_\nu[x]$. For each $i=1, \dots, r$, let $\phi_i(x) \in R_\nu[x]$ be a monic lift of $\overline{\phi}_i(x)$, and $Q_i(x), R_i(x) \in R_\nu[x]$, respectively, the quotient and remainder upon the Euclidean division of $F(x)$ by $\phi_i(x)$. So $Q_i(x)$ is monic and either $R_i(x)=0$ or $\mbox{deg}(R_i(x))<\mbox{deg}(\phi_i(x))$.

Our first theorem (THEOREM \ref{DVR}) gives a precise and easy criterion for the integral closedness of the ring $R_\nu[\al]$ in $L$.

\begin{thm}\label{DVR}
With the above assumptions and notations, $R_\nu[\al]$ is integrally closed in $L$ if and only if, for each $i=1, \dots, r$, either $l_i=1$ or $\nu(R_i(x))=1$.
\end{thm}

For the next result, let $R$ be a Dedekind ring, $K$ its fraction field, $\p$ a nonzero prime ideal of $R$, $\nu_\p$ the (rank-one) discrete valuation of $R$ associated to $\p$, $\displaystyle{F(x) \in R[x]}$ a monic irreducible polynomial, $L=K(\alpha)$ an extension field of $K$ generated by a root $\alpha$ of $F$, $S$ the integral closure of $R$ in $L$, and $k_\p$ the residue field $R/\p$. Keep the same notations and assumptions as above for the factorization of the reduction of $F$ modulo $\p$. 
The following result can be deduced from THEOREM \ref{DVR}, which dramatically and computationally improves Dedekind's criterion in Dedekind ring extensions (see \cite{Er} and \cite{KK2} for instance).


\begin{cor}\label{Ded}
Keep the assumptions and notations of the paragraph above. Then, $S=R[\alpha]$ if and only if, for every prime ideal $\p$ of $R$ whose square divides $\mbox{Disc}_R(\al)$ and for each $i=1, \cdots, r$, either $l_i=1$ or $\nu_\p(R_i(x))=1$.
\end{cor}

In \cite[Theorem 3.1]{HNHA}, it was shown that if $\al$ is a complex
root of an irreducible  polynomial $x^n-m\in \Z[x]$ such that $m$
is square free and every prime divisor of $n$ divides $m$, then
$\Z[\al]$ is integrally closed in $\Q(\al)$. In the following theorem, we give yet an easy new proof of a generalization of the aforementioned result. Note that by saying that an element $a$ of a Dedekind ring $R$ is square-free, we mean that the principal ideal $aR$ is not divisible by the square of any prime ideal of $R$.

\begin{thm}\label{Ded 2}
Let $R$ be a Dedekind ring, $K$ its quotient field, $a\in R$ square-free such that $\displaystyle{f(x)=x^n-a}$ is irreducible over $R$, and $\al$ a root of $f(x)$.
If every prime ideal of $R$ that contains $n.1_K$ also contains $a$, then $R[\al]$ is integrally closed.
\end{thm}

In the case of rings of integers of number fields, the following theorem strongly enhances THEOREM \ref{Ded 2}. Besides, THEOREM \ref{Ded 3} generalizes the relevant results in \cite{KK1} and \cite{Uch}. For a ring of integers $R$, by $\nu_\p(s)$ we mean $\nu_\p(sR)$ for $s\in R$ and a nonzero prime ideal $\p$ of $R$.

\begin{thm}\label{Ded 3}
Let $R$ be the ring of integers of a number field $K$ and
$\displaystyle{L=K(\al)}$ be defined by a root of an irreducible polynomial
$\displaystyle{f(x)=x^n- u \in R[x]}$. Then  $R[\al]$ is integrally closed if and
only if, for every nonzero prime ideal $\p$ of $R$, either of the following holds:

1. $\nu_\p(u)=1$, or

2. $\nu_\p(u)=0$ and $\nu_\p(u^{p^f}-u)=1$, where $p$ is the rational prime lying under $\p$ and $f$ is the residue degree of $\p$ over
$p$.
\end{thm}

If we let $R=\Z$ and $K=\Q$ in THEOREM \ref{Ded 3}, then \cite[Thoerem 1.3]{KK1} can be phrased as follows: $\Z[\al]$ is integrally closed if and only if, for every rational prime $p$, either $\nu_p(u)=0$ or $\nu_p(u)=1$ and $\nu_p(u^{p^{\nu_p(n)}}-u)=1$. The following corollary is an improvement of \cite[Theorem 1.3]{KK1}.

\begin{cor}\label{Z}
Keep the assumptions of THEOREM \ref{Ded 3} with $R=\Z$ and $K=\Q$. Then $\Z[\al]$ is integrally closed if and only if, for every rational prime $p$, either of the following holds:

1. $\nu_p(u)=1$, or

2. $\nu_p(u)=0$ and $\nu_p(u^p-u)=1$.
\end{cor}

\section{{\bf PROOFS OF THE MAIN RESULTS}}

In the notation of THEOREM \ref{DVR}, denote by $\om$ a valuation of $L$ extending $\nu$, by $S_\om$ the valuation ring of $\om$, and by $M_\om$ the maximal ideal of $S_\om$. Note that $S_\nu = \cap_{\om} S_\om$, where the intersection runs over all valuations $\om$ of $L$ extending $\nu$ (see \cite[Lemma 3.17]{Jan}).

We first tackle the following interesting lemma.

\begin{lem}\label{lemma}
Keep the assumptions and notations of THEOREM \ref{DVR}.

(i) For every $1\leq i \leq r$, there exists a valuation $\om$ of $L$ extending $\nu$ such that $\om(\phi_i(\al))>0$.

(ii) For every valuation $\om$ of $L$ extending $\nu$, there exists a unique $1\leq i \leq r$ such that $\om(\phi_i(\al))>0$ and $\om(\phi_j(\al))=0$ for all $j\neq i$.

(iii) For every valuation $\om$ of $L$ extending $\nu$ and every nonzero $p(x)\in R_\nu[x]$, $\om(p(\al))\geq \nu(p(x))$, where equality holds if and only if $\overline{\phi_i}(x)$ does not divide $\overline{(p(x)/\pi^{\nu(p(x))})}$ for some $\phi_i(x)$ satisfying $\om(\phi_i(\al))>0$.
\end{lem}

\begin{proof}$\\$

(i) We know (see \cite[Proposition III.8.2]{Neu}) that the valuations $\om_1, \cdots, \om_t$ of $L$ extending $\nu$ are in one-to-one correspondence with the irreducible factors $F_1(x), \cdots, F_t(x)$ of $F(x)$ in $K_\nu[x]$, where $K_\nu$ is the separable closure of $K$ in the $\nu$-adic completion of $K$. Moreover, if $\overline{\nu}$ is the unique valuation extending $\nu$ to the algebraic closure of $K_\nu$, then for any $h(x)\in K[x]$ and every root $\al_j$ of $F_j(x)$, $\om_j(h(\al))=\overline{\nu}(h(\al_j))$. Now choose $j$ so that $\al=\al_j$. By Hensel's Lemma, there is some $i$ such that $\overline{F_j}(x)$ is a power of $\overline{\phi_i}(x)$ (say $\overline{\phi_i}(x)^{e_i}$) modulo $\m_\nu$. As $F_j(\al)=0$ and $\m_\nu[\al] \subset M_{\om_j}$, $\phi_i(\al)^{e_i}\in M_{\om_j}$. Thus $\om_j(\phi_i(\al))>0$ as claimed.

(ii) Let $\om$ be a valuation of $L$ extending $\nu$. Assume for the moment that the first assertion of part (iii) is true. Since $\prod_{i=1}^r \phi_i(\al)^{l_i} \equiv f(\al) \equiv 0 \; (\mbox{mod} \; M_\om)$, $\om(\prod_{i=1}^r \phi_i(\al)^{l_i}) >0$. So, $\om(\phi_i(\al)) >0$ for some $1 \leq i \leq r$. For any $j\neq i$, let $s_j(x), t_j(x) \in R_\nu[x]$ be such that $\overline{s_j}(x) \overline{\phi_i}(x) + \overline{t_j}(x)\overline{\phi}_j(x)\equiv 1 \;(\mbox{mod} \;\m_\nu)$. Then, $s_j(\al)\phi_i(\al) + t_j(\al)\phi_j(\al) = 1 +h(\al)$ for some $h(x) \in \m_\nu[x]$. As $\nu(h(x))>0$, it follows from the first assertion of part (iii) that $\om(h(\al))>0$ and, thus, $h(\al) \in M_\om$. Since $\phi_i(\al)\in M_\om$ (because $\om(\phi_i(\al))>0$) and $s_j(\al)\in R_\nu[\al]\subseteq S_\nu \subseteq S_\om$, $s_j(\al) \phi_i(\al))\in M_\om$.  Thus, $t_j(\al)\phi_j(\al) -1 \in M_\om$ and, therefore, $t_j(\al)\phi_j(\al)\in S_\om - M_\om$. So $\om(t_j(\al)\phi_j(\al))=0$ and, thus, $\om(\phi_j(\al)) =0$, and the uniqueness of $i$ such that $\om(\phi_i(\al))>0$ follows.

(iii) Let $\om$ be a valuation of $L$ extending $\nu$, $p(x)\in R_\nu[x]$ be nonzero, and set $p_1(x)=p(x)/\pi^u$, where $u=\nu(p(x))$. As $\nu(p_1(x))=0$, $p_1(x) \in R_\nu[x]$. Thus, $p_1(\al)\in S_\nu\subseteq S_\om$ and $\om(p(\al))=\om(\pi^u p_1(\al))=u+\om(p_1(\al)) \geq u$ as claimed. Now define the map $\psi: k_\nu[x] \to S_\om/M_\om$ by $\overline{p}(x) \mapsto p(\al)+ M_\om$. This is a well-defined map since $\m_\nu \subseteq M_\om$. It can also be checked that $\psi$ is a ring homomorphism. For a nonzero $p(x)\in R_\nu[x]$ and $p_1(x)=p(x)/\pi^u$ with $u=\nu(p(x))$, we have $\om(p(\al))=u +\om(p_1(\al))$. So, $\om(p(\al))=u$ if and only if $\om(p_1(\al))=0$ if and only if $p_1(\al)\in S_\om - M_\om$ if and only if $\overline{p_1}(x) \not\in \mbox{ker}\,\psi$. From part (ii), let $\phi_i(x)$ be such that $\om(\phi_i(\al))>0$. Then, $\phi_i(\al)\in M_\om$ and, thus, $\overline{\phi_i}(x) \in \mbox{ker}\,\psi$. Since $\mbox{ker}\,\psi$ is principal (as $k_\nu$ is a field) and $\overline{\phi_i}(x)$ is irreducible over $k_\nu$, $\mbox{ker}\,\psi$ is generated by $\overline{\phi_i}(x)$. It now follows that $\om(p(\al))=u$ if and only if $\overline{\phi_i}(x)$ does not divide $\overline{p_1}(x)$ as claimed.
\end{proof}

\begin{proof} ({\bf THEOREM \ref{DVR}})
We prove first that if $R_\nu[\al]$ is integrally closed in $L$, then $l_i=1$ or $\nu(R_i(x))=1$ for each $i=1, \dots, r$. Assume that there exists some $k \in \{1, \dots, r\}$ such that $l_k>1$ and $\nu(R_k(x))>1$. Set $$\theta_k = Q_k(\al)/\pi = -R_k(\al)/(\pi \phi_k(\al));$$ we show that $\theta_k$ is an element of $S_\nu-R_\nu[\al]$ and, thus, $R_\nu[\al]$ is not integrally closed. Since $Q_i(x)$ is monic, $\theta_k \not\in R_\nu[\al]$ as, otherwise, $1/\pi$ would be an element of $R_\nu$, which is absurd. To show that $\theta_k\in S_\nu$, we show that $\theta_k \in S_\om$ for each valuation $\om$ of $L$ extending $\nu$ (as $S_\nu = \cap_{\om} S_\om$). Let $\om$ be such a valuation. By LEMMA \ref{lemma} (ii), let $i\in \{1, \dots, r\}$ be such that $\om(\phi_i(\al))>0$ and $\om(\phi_j(\al)=0$ for all $j\neq i$. Note, by LEMMA \ref{lemma} (iii), that $\om(R_k(\al)) \geq \nu(R_k(x))>1$. If $k \neq i$, then $$\om(Q_k(\al)) = \om(\phi_k(\al)) + \om(Q_k(\al)) =\om(\phi_k(\al) Q_k(\al))=\om(R_k(\al))>1.$$ So, $\om(\theta_k)=\om(Q_k(\al))-1>0$. Thus, $\theta_k \in S_\om$ in this case. If $k=i$, we consider two possibilities. If $0<\om(\phi_k(\al)) \leq 1$, then
$$\om(\theta_k)=\om(R_k(\al))-\om(\pi)-\om(\phi_k(\al)) \geq 2-1-1=0.$$ So, $\theta_k \in S_\om$ in this case too.
If, on the other hand, $\om(\phi_k(\al)) >1$, let $q_k(x), \, r_k(x)\in R_\nu[x]$ be, respectively, the quotient and remainder upon the Euclidean division of $Q_k(x)$ by $\phi_k(x)$ with $q_k(x)$ monic. We now have $$\overline{F}(x)\equiv \overline{q_k}(x)\overline{\phi_k}^2(x)+\overline{r_k}(x)\overline{\phi_k}(x)+\overline{R_k}(x) \; (\mbox{mod}\, m_\nu).$$ Since $\overline{\phi_k}^2 (x)$ divides $\overline{F}(x)$ (as $l_k\geq 2$) and $\overline{R_k}(x)\equiv 0 \; (\mbox{mod}\, m_\nu)$, it follows that $\overline{\phi_k}^2(x)$ divides $\overline{r_k}(x) \overline{\phi_k}(x)$ and, hence, $\overline{\phi_k}(x)$ divides $\overline{r_k}(x)$. As $\mbox{deg}(\overline{r_k}(x)) < \mbox{deg}(\overline{\phi_k}(x))$, $\overline{r_k}(x) \equiv 0\; (\mbox{mod}\, m_\nu)$ and $\nu(r_k(x))\geq 1$. Now (using LEMMA \ref{lemma} (iii) in the third inequality below), we have
\begin{align*}
\om(Q_k(\al))&=\om(q_k(\al)\phi_k(\al) + r_k(\al))\\
& \geq \mbox{min}\{\om(q_k(\al))+\om(\phi_k(\al)), \om(r_k(\al))\}\\
&\geq \mbox{min}\{\om(\phi_k(\al)), \om(r_k(\al))\}\\
&\geq \mbox{min}\{\nu(\phi_k(x), \nu(r_k(x))\}\\
&\geq 1.
\end{align*}
Thus, $\om(\theta_k)=\om(Q_k(\al))-1 \geq 0$ and, hence, $\theta_k \in S_\om$ in this case as well.

For the converse, assume that for every $1\leq i \leq r$, either $l_i=1$ or $\nu(R_i(x))=1$. We proceed in three steps.

\underline{Step 1:} We show that if, for some $i$, $l_i=1$, then we can always assume that $\nu(R_i(x))=1$ too. Suppose that $\nu(R_i(x))>1$. Note that
$$F(x)=Q_i(x)\phi_i(x)+R_i(x)=Q_i(x)(\phi_i(x)+\pi)-\pi Q_i(x) + R_i(x).$$
Let $H_i(x), T_i(x)\in R_\nu[x]$ be such that $Q_i(x)=H_i(x) \phi_i(x) + T_i(x)$ with $\mbox{deg}(T_i(x)) < \mbox{deg}(\phi_i(x))$. Set $\phi_i^*(x)=\phi_i(x) + \pi$, $Q_i^*(x)=Q_i(x)-\pi H_i(x)$ and $R_i^*(x)= R_i(x)-\pi T_i(x)$. Then, $F(x)= Q_i^*(x) \phi_i^*(x)+R_i^*(x)$. Note that $Q_i^*(x)$ and $R_i^*(x)$ are, respectively, the quotient and remainder upon the Euclidean division of $F(x)$ by $\phi_i^*(x)$. Since $\nu(R_i(x)) >1$ and $\overline{T_i}(x)$ is nonzero (as $l_i=1$), it must follow that $\nu(R_i^*(x))=1$. So, up to replacing the lifting of $\overline{\phi_i}(x)$ by $\phi_i^*(x)$ instead of $\phi_i(x)$ if necessary, we can assume that $\nu(R_i(x))=1$ as claimed.

\underline{Step 2:} Based on Step 1, we can assume that $\nu(R_i(x))=1$ for every $1\leq i \leq r$. Let $\om$ be a valuation of $L$ extending $\nu$ and $i\in \{1, \cdots, r\}$, we show that if $\om(\phi_i(\al))>0$ then $\om(\phi_i(\al))=1/{l_i}$. If $\l_i=1$, then $\overline{\phi_i}(x)$ does not divide $\overline{Q_i}(x)$. So, by LEMMA \ref{lemma} (iii), $\om(Q_i(\al))=0$ and $\om(\phi_i(\al))=\om(Q_i(\al) \phi_i(\al))=\om(-R_i(\al))=\nu(R_i(x))=1=1/{l_i}$. If $l_i >1$, then set $$F(x)=G_i(x)\phi_i^{l_i}(x) + S_i(x) \phi_i(x) + R_i(x),$$ for some $G_i(x), S_i(x)\in R_\nu[x]$ with $\nu(G_i(x))=0$ and $\nu(S_i(x))>1$. It then follows that $\om(G_i(\al) \phi_i^{l_i}(\al))=\om(S_i(\al)\phi_i(\al)+R_i(\al))=1$. Thus, $\om(\phi_i^{l_i}(\al) =1$ and, therefore, $\om(\phi_i(\al))=1/{l_i}$.

\underline{Step 3:} Now assume that $R_\nu[\al]$ is not integrally closed. So, there exists some monic $p(x)\in R_\nu[x]$ with $\mbox{deg}(p(x)) < \mbox{deg}(F(x))$ such that $p(\al)/\pi$ is integral over $R_\nu$. Note then that $p(\al)/\pi \in S_\nu - R_\nu[\al]$. Let $r_i\geq 0$ be such that $\overline{\phi_i}^{r_i}(x)$ is the highest power of $\overline{\phi_i}(x)$ that divides $\overline{p}(x)$. Since $\mbox{deg}(p(x))<\mbox{deg}(F(x))$, $r_{i_0} < l_{i_0}$ for some $i_0\in \{1, \cdots, r\}$. Let $M_{i_0}(x), L_{i_0}(x) \in R_\nu[x]$ be the quotient and remainder upon the Euclidean division of $p(x)$ by $\phi_{i_0}(x)$. So $$p(x)=\phi_{i_0}^{i_0}(x)M_{i_0}(x)+L_{i_0}(x),$$ $\overline{\phi_{i_0}}(x)\nmid \overline{M_{i_0}}(x)$, and $\nu(L_{i_0}(x))\geq 0$. Since $p(x)$ and $\phi_{i_0}(x)$ are monic, $M_{i_0}(x)$ is monic and, therefore, $\nu(M_{i_0}(x))=0$. By LEMMA \ref{lemma} (i), let $\om$ be a valuation of $L$ extending $\nu$ such that $\om(\phi_{i_0}(\al))>0$. Then, by Step 2 above, $om(\phi_{i_0}(\al))=1/{l_{i_0}}$. Since $\overline{\phi_{i_0}}(x)\nmid \overline{M_{i_0}}(x)$ and $\nu(M_{i_0}(x))=0$, it follows from LEMMA \ref{lemma} (iii) that $\om(M_{i_0}(\al))=\nu(M_{i_0}(x))=0$. Now, as $r_{i_0}/l_{i_0} \neq \om(L_{i_0}(\al))$,
\begin{align*}
\om(p(\al))&=\mbox{min}\{\om(\phi_{i_0}^{r_{i_0}}(\al)M_{i_0}(\al)), \om(L_{i_0}(\al))\}\\
& =\mbox{min}\{r_{i_0}\om(\phi_{i_0}(\al)+M_{i_0}(\al)), \om(L_{i_0}(\al))\}\\
& = \mbox{min}\{r_{i_0}/l_{i_0}, \om(L_{i_0}(\al))\}\\
& \leq r_{i_0}/l_{i_0} <1.
\end{align*}
Thus, $\om(p(\al)/\pi)<0$ and, therefore, $p(\al)/\pi \not\in S_\nu$, a contradiction.
\end{proof}

\noindent{\it Remark.} Checking that $\nu(R_i(x))=1$ is needed only if $l_i \geq 2$. In this case, note that the requirement that $\nu(R_i(x))=1$ is independent of the choice of the monic lifting of the $\ol{\phi_i}(x)$. Indeed, if $l_i \geq 2$, then we show that $\nu(R_i(x))=1$ if and only if $\nu(r_i(x))=1$ for the remainder $r_i(x)$ upon the Euclidean division of $f(x)$ by any other monic lifting of $\ol{\phi_i}(x)$. Let $\nu(R_i(x))=1$, $P_i(x)=\phi_i(x)+\pi H(x)$ be another monic lifting of $\ol{\phi_i}(x)$, with $H(x)\in R_\nu[x]$. Let $q_i(x), r_i(x)\in R_\nu[x]$ be, respectively, the quotient and remainder upon the Euclidean division of $f(x)$ by $P_i(x)$. Let $\omega$ be a valuation of $L$ extending $\nu$ such that $\omega(\phi_i(\al))>0$. By the proof of THEOREM \ref{DVR}, $\omega(\phi_i(\al))=1/l_i$. We have $$f(x) =Q_i(x)\phi_i(x)+R_i(x)=q_i(x)P_i(x)+r_i(x).$$
As $\omega(\phi_i(\al))=1/l_i <1$ and $\omega(\pi H(\al)) \geq 1$, $\omega(P_i(\al))=\omega(\phi_i(\al)+\pi H(\al))=1/l_i$. Let $M_i(x), S_i(x) \in R_\nu[x]$ be, respectively, the quotient and remainder upon the Euclidean division of $f(x)$ by $P_i^{l_i}(x)$. Since $\ol{\phi_i}(x)$ does not divide $\ol{M_i}(x)$, $\omega(M_i(\al))=\nu(M_i(x))=0$ (by LEMMA \ref{lemma}) and, thus, $\omega(S_i(\al))=\omega(M_i(\al)P_i^{l_i}(\al))=1$. Note that $S_i(x)=N_i(x) P_i(x) + r_i(x)$ for some $N_i(x) \in \m[x]$. So, $\omega(r_i(\al))=\omega(S_i(\al))=1$ because $\omega(N_i(\al)P_i(\al))\geq 1+1/l_i > 1$. Since $\ol{\phi_i}(x)$ does not divide $\ol{r_i}(x)$, $\nu(r_i(x))=\omega(r_i(\al))=1$.

\begin{proof} ({\bf COROLLARY \ref{Ded}})
On the one hand, it is known that $R[\al]$ is integrally closed (i.e. $S=R[\al]$) if and only if $R_\p[\al]$ is integrally closed for every nonzero prime ideal $\p$ of $R$ (see \cite{Cha-Dea}). On the other hand, the generalized discriminant-index formula "$\displaystyle{\mbox{Disc}_R(F)=\mbox{Ind}_R(\al)^2 \mbox{D}_R(S)}$" was shown in \cite{Cha-Dea 2}. It is thus obvious that for the equality $S=R[\al]$ to hold, we need only to consider those prime ideals $\p$ of $R$ whose squares divide $\mbox{Disc}_R(F)$.
For such a prime ideal, $(K, \nu)$ is a valued field of rank 1 with discrete valuation $\nu_\p$ and ring of valuation $R_\p$. Applying THEOREM \ref{DVR} yields the desired conclusion.
\end{proof}


\begin{proof} ({\bf THEOREM \ref{Ded 2}})
In order to use COROLLARY \ref{Ded}, and since $\mbox{disc}(f)=\pm n^n a^{n-1}$, we need only to consider those prime ideals of $R$ containing $n.1_K$ or $a$. Since any prime ideal $\p$ that contains $\mbox{disc}(f)$ must contain $a$ (by our assumption on $n$), we let $\p$ be a prime ideal of $R$ containing $a$, Then $x^n-a\equiv x^n \,(\mbox{mod}\, \p)$. By the Euclidean division of $x^n-a$ by $x$, the remainder is $-a$. Since $a$ is square-free, $\nu_\p(-a)=1$. Thus, by COROLLARY \ref{Ded}, $R[\al]$ is integrally closed in $K(\al)$.
\end{proof}

\begin{proof}({\bf THEOREM \ref{Ded 3}})
It is known that $R[\al]$ is integrally closed in $L$ if and only if $R_\p[\al]$ is integrally closed in $L$ for every nonzero prime ideal $\p$ of $R$ that divides the the discriminant of $f(x)$. Since the discriminant of $f(x)$ is $n^nu^{n-1}$, we seek to show the integral closedness of $R_\p[\al]$ in $L$ for every nonzero prime ideal $\p$ of $R$ that contains $nu$. Let $\p$ be such a prime ideal. If $u\in \p$, then it follows from THEOREM \ref{Ded 2} that $R_\p[\al]$ is integrally closed in $L$ if and only if $u\not\in\p^2$. Assume that $u\not\in \p$. So, $n1_K\in \p$ and $n$ is thus divisible by $p$. Let $n=m p^r$ with $m\in \N$ not divisible by $p$. If, on the one hand, $f \leq r$, then set $r=s+f$. So, $$f(x)=x^{m p^r}-u \equiv x^{m p^s p^f}-u^{p^f} \equiv (x^{m p^s}-u)^{p^f} \;\;(\mbox{mod}\;\p).$$
We also have,
\begin{align*}
f(x)&=(x^{m p^s})^{p^f} -u = (x^{m p^s}-u +u)^{p^f}-u\\ &=\sum_{k=0}^{p^f-1} \left(\begin{array}{c} p^f \\ k \\ \end{array} \right) u^k (x^{m p^s}-u)^{p^f -k} +u^{p^f} -u \\ &= H(x) (x^{m p^s}-u) + u^{p^f}-u,
\end{align*}
$H(x) \in R[x]$. If $\overline{x^{m p^s}-u} = \prod_{i=1}^t \overline{g_i}^{e_i}(x)$ is the monic irreducible factorization of $x^{m p^s} - u$ module $\p$, then $\overline{f}(x)=\prod_{i=1}^t \overline{g_i}^{e_i p^f}(x)$ is the monic irreducible factorization of $f(x)$ modulo $\p$. Letting $g_i(x) \in R[x]$ be a monic lift of $\overline{g_i}(x)$ for each $i$, it follows that the remainder upon the Euclidean division of $f(x)$ by each $g_i(x)$ is $u^{p^f}-u$. If, on the other hand, $r<f$, then set $f=s+r$. So, $$f(x)=x^{m p^r}-u \equiv x^{m p^r}-u^{p^f} \equiv (x^m -u^{p^s})^{p^r} \;\;(\mbox{mod}\;\p).$$
We also have,
\begin{align*}
f(x)&=(x^m)^{p^r} -u = (x^m -u^{p^s} +u^{p^s})^{p^r}-u\\ &=\sum_{k=0}^{p^r-1} \left(\begin{array}{c} p^r \\ k \\ \end{array} \right) u^k p^s (x^m -u^{p^s})^{p^r -k} +u^{p^f} -u \\ &= M(x) (x^m-u^{p^s}) + u^{p^f}-u,
\end{align*}
$M(x) \in R[x]$. If $\overline{x^m-u^{p^s}} = \prod_{i=1}^v \overline{h_i}^{l_i}(x)$ is the monic irreducible factorization of $x^m - u^{p^s}$ module $\p$, then $\overline{f}(x)=\prod_{i=1}^v \overline{h_i}^{l_i p^r}(x)$ is the monic irreducible factorization of $f(x)$ modulo $\p$. Letting $h_i(x) \in R[x]$ be a monic lift of $\overline{h_i}(x)$ for each $i$, it follows that the remainder upon the Euclidean division of $f(x)$ by each $h_i(x)$ is $u^{p^f}-u$. In either case, it follows from THEOREM \ref{DVR} that $R_\p[\al]$ is integrally closed if and
only if $\nu_\p(u^{p^f}-u)=1$.
\end{proof}

\noindent{\it Remark.} In case (ii) of THEOREM \ref{Ded 3}, any $r\in \N$ with $\nu_\p(u^{p^r} - u)=1$ suffices for the same conclusion to hold.\\

\begin{proof}({\bf COROLLARY \ref{Z}})
Just apply THEOREM \ref{Ded 3} noting that $f=1$.
\end{proof}





\section{{\bf APPLICATIONS AND EXAMPLES}}



\begin{cor}\label{cor 1} Keep the notations of THEOREM \ref{DVR} and let $f(x)=\sum_{i=0}^n a_i x^i \in R_\nu[x]$ be monic with $\nu(a_k) \geq 1$ for $1\leq k \leq n-1$, and $\nu(a_0) =1$. Let $L=K(\al)$ a field extension of $K$ with $\al$ a root of $f(x)$. Then $f(x)$ is irreducible over $K$ and $R_\nu[\al]$ is integrally closed in $L$.
\end{cor}

\begin{proof}
By the well-known Eisenstein's Criterion, $f(x)$ is irreducible over $R_\nu$. By Gauss' Lemma (see \cite[Proposition 9.4.5]{Dummit}), $f(x)$ is also irreducible over $K$ as well. As $\overline{f}(x)\equiv x^n \;(\mbox{mod}\;\m_\nu)$ and the remainder when dividing $f(x)$ by $x$ is $a_0$ with $\nu(a_0)=1$, it follows from THEOREM \ref{DVR} that $R_\nu[\al]$ is integrally closed in $L$.
\end{proof}

With the notation of THEOREM \ref{DVR}, assume that $f(x), \phi(x) \in R_\nu[x]$ are monic polynomials such that $\mbox{deg}(\phi(x))\leq \mbox{deg}(f(x))$ and $\overline{\phi}(x) \in k[x]$ is monic and irreducible. Let $f(x) =\sum_{i=0}^l a_i(x) \phi(x)^{l-i}$ be the $\phi$-adic expansion of $f(x)$. This entails, in particular, that, for each $i$, either $a_i(x)=0$ or $\mbox{deg}(a_i(x))<\mbox{deg}(\phi(x))$. We say that $f(x)$ is $(\phi, \nu)$-Eisenstein if $\nu(a_i(x)) \geq 1$ for $i=1, \dots, l-1$ and $\nu(a_l(x))=1$.

\begin{cor}\label{phi-Eis}
Consider the above notation and assumptions, and let $L=K(\al)$ be a field extension of $K$ with $\al$ a root of $f(x)$.  If $\overline{f}(x) \equiv \overline{\phi}(x)^l \; (\mbox{mod} \; \m_\nu)$ and $f(x)$ is $(\phi, \nu)$-Eisenstein, then $f(x)$ is irreducible over $R_\nu$ and $R_\nu[\al]$ is integrally closed in $L$.

\end{cor}

\begin{proof}
Assume that $f(x)=g(x) h(x)$ for some monic $g(x), h(x)\in R_\nu[x]$. Then $\overline{g}(x) \equiv \overline{\phi}(x)^{l_1},\, \overline{h}(x) \equiv \overline{\phi}(x)^{l_2} \; (\mbox{mod}\; \m_\nu)$, with $l_1 + l_2 =l$. Note that $l_i\geq 1$ for $i=1,2$ as both $g(x)$ and $h(x)$ are monic. Let $g(x)=\sum_{i=0}^{l_1} g_i(x) \phi(x)^{l_1 -i}$ and $h(x) =\sum_{i=0}^{l_2} h_i(x) \phi(x)^{l_2 -i}$ be the $\phi$-adic expansions of $g(x)$ and $h(x)$, respectively. As $g(x)$ and $h(x)$ are monic, $g_0(x) = h_0(x)=1$. Since $l_i\geq 1$ and $\overline{g}(x) \equiv \overline{\phi}(x)^{l_1},\, \overline{h}(x) \equiv \overline{\phi}(x)^{l_2} \; (\mbox{mod}\; \m_\nu)$, $\nu(g_{l_1}(x))\geq 1$ and $\nu(h_{l_2}(x))\geq 1$. By the uniqueness of the $\phi$-adic expansion of $f(x)$, $a_l(x)=g_{l_1}(x)h_{l_2}(x)$. Thus, $\nu(a_l(x)) \geq 2$, a contradiction. Thus, $f(x)$ is irreducible over $R_\nu$. Hence, $f(x)$ is irreducible over $R_\nu$. Now, since $\nu(a_l(x))=1$, it follows from THEOREM \ref{DVR} that $R_\nu[\al]$ is integrally closed in $L$.
\end{proof}


\begin{cor}
Consider the above notation and assumptions, and let $f(x) = x^n +a \in R_\nu[x]$ be monic such that $\nu(a)=m \geq 1$ with $m$ and $n$ relatively prime. Let $L=K(\al)$ be a field extension of $K$ with $\al$ a root of $f(x)$ and $S$ the integral closure of $R_\nu$ in $L$. Then $f(x)$ is irreducible over $K$ and $\theta=\al^s/\pi^t$ generates a power basis for $S$ over $R_\nu$, where $s,t\in\Z$ such that $ms-nt=1$ .
\end{cor}

\begin{proof}
As $\theta^n=\al^{ns}/\pi^{nt}=a^s/\pi^{nt}$, $\nu(\theta^n)=ms-nt=1$. So, $F(x)=x^n-\theta^n \in R_\nu[x]$ is $\nu$-Eisenstein. Hence, by COROLLARY \ref{cor 1}, $F(x)$ is irreducible over $K$ and $R_\nu[\theta]$ is integrally closed in $K(\theta)$. But $\theta=\al^s/\pi^t \in L$. On the other hand, $\al=\al^{ms-nt}=\al^{ms}/\al^{nt}=(\pi^{mt}\theta^m)/a^t \in K(\theta)$. Thus, $L=K(\theta)$. This, on the one hand, implies that $R_\nu[\theta]$ is integrally closed in $L$ as claimed. On the other hand, as $f(x)$ and $F(x)$ are of the same degree and $F(x)$ is irreducible over $K$, $f(x)$ is irreducible over $K$ as well.
\end{proof}




\noindent{\it Example 1.}

In this example we use COROLLARY \ref{Ded} to give a much easier proof of the very well-know monogenity of $n$th cyclotomic number fields. By
\cite[p. 11]{Wash}, it suffices to prove the monogenity of $p^r$th cyclotomic number fields for rational primes $p$.
Let $K_{p^r}=\Q(\zeta)$ be the $p^r$th cyclotomic field with $\zeta=\zeta_{p^r}=\exp(2\pi i/p^r)$. It is known that the minimal polynomial of $\zeta$ is $$\Phi_{p^r}(x)=\frac{x^{p^r}-1}{x^{p^{r-1}}-1}=x^{\phi(p^r)}+x^{\phi(p^r)-p^{r-1}}+\cdots + x^{\phi(p^r)-(p-2)p^{r-1}}+1$$ and $p$ is the only rational prime whose square divides $\mbox{disc}(\Phi_{p^r})$ (in fact, $\mbox{disc}(\Phi_{p^r})$ is a power of $p$). Reducing $\Phi_{p^r}(x)=\cfrac{x^{p^r}-1}{x^{p^{r-1}}-1}$ modolo $p$ yields
$$\overline{\Phi_{p^r}}(x)\equiv \overline{(x-1)}^{\;\phi(p^r)} \;\;(\mbox{mod}\;p).$$
Let $Q(x), R(x)\in \Z[x]$ be, respectively, the quotient and remainder upon the Euclidean division of $\Phi_{p^r}(x)$ by $x-1$. Since $\mbox{deg}(x-1)=1$, $R(x)=a$ for some constant $a\in \Z$. Thus, $\Phi_{p^r}(x)=(x-1)Q(x)+a$. Evaluating both sides at 1 yields $p=\Phi_{p^r}(1)=a$. Since $\nu_p(p)=1$, it now follows from COROLLARY \ref{Ded} that $\Z_{K_{p^r}}=\Z[\zeta]$, which is what we need to show.\\

\noindent{\it Example 2.}

Let $R=\Z_K$, where $K$ is the quadratic number field defined by $x^2-3$. It
is well known that $R=\Z[\sqrt{3}]$ and $3R=\p^2$, where
$\p=\sqrt{3} R$. Let $(m,n)\in \Z\times \N$ be two integers such
$f(x)=x^n-m$ is irreducible over $K$ and $3$ divides $n$. Let
$L=K(\al)$, where  $\al$ is a root of $f(x)$. We show that $R[\al]$
is not integrally closed. If $3$ divides $m$, then as $3$ is a square in $R$, $m$ is
not square free in $R$. So, by THEOREM \ref{Ded 3}, $R[\al]$ is not
integrally closed. If $3$ does not divide $m$, then as $m^3\equiv m \, \mbox{mod } 3$,
$3$ divides $m^3-m$ in $R$ and, thus, $\nu_{\p}(m^3-m)\ge 2$. Again, by THEOREM \ref{Ded 3}, $R[\al]$ is not
integrally closed in this case either.



\begin{thebibliography} {Dillo 83}
\bibitem{Cha-Dea}
{\sc M. E. Charkani} and {\sc A. Deajim}, {\em Generating a power
basis over a Dedekind ring}, J. Number Theory {\bf 132} (2012), 2267--2276.

\bibitem{Cha-Dea 2}
{\sc M. E. Charkani} and {\sc A. Deajim}, {\em Relative index in extensions of Dedekind rings}, JP J. of Algebra, Number Theory, and Appl. {\bf 27} (2012), 73--84.


\bibitem {Cohen}
{\sc H. Cohen}, {\em A Course in Computational Algebraic Number
Theory}, GTM vol. 138, Springer Verlag, Berlin, 1996.

\bibitem{Ded}
{\sc R. Dedekind}, {\em Uber den zussamenhang zwischen der theorie
der ideals und der theorie der hoheren cyclotimy index}, Abh. Akad.
Wiss. Gottingen, Math.-Phys. KL {\bf 23} (1878) 1--23.



\bibitem{Er}
{\sc Y. Ershov}, {\em A Dedekind criterion for arbitrary valuation rings}, Dokl. Math. {\bf 74}(2) (2006), 650--652.

\bibitem{Jan}
{\sc G. Janusz}, { \em  Algebric Number Fields } (Academic Press,
New York,  Second Edition 1995).

\bibitem{KK1}
{\sc K. Khanduja} and {\sc B. Jhorar}, {\em When is $R[\theta]$ integrally closed?}, J. Algebra and Appl. {\bf 15}(5) (2016).

\bibitem{KK2}
{\sc K. Khanduja} and {\sc M. Kumar}, {\em A generalizaiton of Dedekind Criterion}, Commun. Algebra {\bf 35}(5) (2007), 1479--1486.

\bibitem{HNHA}
{\sc A. Hameed}, {\sc T. Nakahara}, {\sc S. Husnine}, and {\sc S. Ahmad}, {\em On existence of canonical number system in certain classes of pure algebraic number fields}, J. Prime Res. Math {\bf 7} (2011), 19--24.


\bibitem{Neu}
{\sc J.  Neukirch}, {\em Algebraic Number Theory}, Springer-Verlag, Berlin, 1999.

\bibitem{Sch}
{\sc P. Schmid}, {\em On criteria by Dedekind and Ore for integral
ring extensions}, Arch. Math. {\bf 84} (2005), 304--310.

\bibitem{Uch}
{\sc K. Uchida}, {\em When is $\mathbb{Z}[\al]$ the ring of the integers?}, Osaka J. Math {\bf 14} (1977), 155--157.
Osaka J. Math. 14 (1977) 155–157]

\bibitem{Wash}
{\sc L. Washington}, {\em Introduction to Cyclotomic Field}, Springer-Verlag, New York, 1997.
\end{thebibliography}
\end{document}